\documentclass{amsart}
\usepackage{amssymb,
enumitem,
mathrsfs,
mathtools,
tikz,
upgreek,
verbatim,
hyphenat,
url
}
\usepackage[T1]{fontenc}
\usepackage[shadow]{todonotes}
\usepackage{tikz-cd}
 
\usepackage{newpxtext,newpxmath}

\usetikzlibrary{calc}

\usetikzlibrary{shapes.misc, positioning}

\usepackage[colorlinks=true,linkcolor={black},citecolor={blue}]{hyperref}%

\renewcommand{\wr}{\mathbin{\mathrm{wr}}}

\newcommand{\LO}{\mathrm{LO}}
\newcommand{\CLO}{\mathrm{CLO}}
\newcommand{\Diff}{\mathrm{Diff}}
\newcommand{\lex}{\textit{lex\,}}

\newenvironment{enumerate-(a)}{\begin{enumerate}[label={\upshape (\alph*)}, leftmargin=2pc]}{\end{enumerate}}
\newenvironment{enumerate-(a)-r}{\begin{enumerate}[label={\upshape (\alph*)}, leftmargin=2pc,resume]}{\end{enumerate}}
\newenvironment{enumerate-(a)-5}{\begin{enumerate}[label={\upshape (\alph*)}, leftmargin=2pc,start=5]}{\end{enumerate}}
\newenvironment{enumerate-(A)}{\begin{enumerate}[label={\upshape (\Alph*)}, leftmargin=2pc]}{\end{enumerate}}
\newenvironment{enumerate-(A)-r}{\begin{enumerate}[label={\upshape (\Alph*)}, leftmargin=2pc,resume]}{\end{enumerate}}
\newenvironment{enumerate-(i)}{\begin{enumerate}[label={\upshape (\roman*)}, leftmargin=2pc]}{\end{enumerate}}
\newenvironment{enumerate-(i)-r}{\begin{enumerate}[label={\upshape (\roman*)}, leftmargin=2pc,resume]}{\end{enumerate}}
\newenvironment{enumerate-(I)}{\begin{enumerate}[label={\upshape (\Roman*)}, leftmargin=2pc]}{\end{enumerate}}
\newenvironment{enumerate-(I)-r}{\begin{enumerate}[label={\upshape (\Roman*)}, leftmargin=2pc,resume]}{\end{enumerate}}
\newenvironment{enumerate-(1)}{\begin{enumerate}[label={\upshape (\arabic*)}, leftmargin=2pc]}{\end{enumerate}}
\newenvironment{enumerate-(1)-r}{\begin{enumerate}[label={\upshape (\arabic*)}, leftmargin=2pc,resume]}{\end{enumerate}}

\newtheorem{theorem}{Theorem}[section]
\newtheorem{lemma}[theorem]{Lemma}
\newtheorem{corollary}[theorem]{Corollary}
\newtheorem{proposition}[theorem]{Proposition}
\newtheorem{question}[theorem]{Question}

\newtheorem{fact}[theorem]{Fact}

\theoremstyle{definition}
\newtheorem{definition}[theorem]{Definition}

\newtheorem{example}[theorem]{Example}
\newtheorem{conjecture}[theorem]{Conjecture}
\theoremstyle{remark}
\newtheorem{remark}[theorem]{Remark}

\newtheorem{innercustomgeneric}{\customgenericname}
\providecommand{\customgenericname}{}
\newcommand{\newcustomtheorem}[2]{%
  \newenvironment{#1}[1]
  {%
   \renewcommand\customgenericname{#2}%
   \renewcommand\theinnercustomgeneric{##1}%
   \innercustomgeneric
  }
  {\endinnercustomgeneric}
}

\newcustomtheorem{customthm}{Theorem}
\newcustomtheorem{customlemma}{Lemma}

\begin{document}

\title[The Borel complexity of the space of left-orderings]{The Borel complexity of the space of left-orderings, low-dimensional topology, and dynamics}

\date{}
\author[F.~Calderoni]{Filippo Calderoni}

\address{Department of Mathematics, Rutgers University, 
Hill Center for the Mathematical Sciences,
110 Frelinghuysen Rd.,
Piscataway, NJ 08854-8019}
\email{filippo.calderoni@rutgers.edu}

\author[A.~Clay]{Adam Clay}

\address{Department of Mathematics, 420 Machray Hall, University of
Manitoba, Winnipeg, MB, R3T 2N2, Canada}
\email{Adam.Clay@umanitoba.ca}

\thanks{The paper was completed during the Thematic Program in Geometric Group Theory held at the Centre de Recherche Math\'ematiques at the Universit\'e de Montr\'eal in 2023. We would like to thank Steve Boyer for valuable comments about Section~5, and Tommy Cai for helping with Proposition~5.3. Calderoni's research was partially supported by the NSF grant DMS--2348819. Clay's research was partially supported by NSERC grant RGPIN-2020-05343.}


\maketitle

\begin{abstract}
We develop new tools to analyze the complexity of the conjugacy equivalence relation \(E_\mathsf{lo}(G)\), whenever $G$ is a left-orderable group.  Our methods are used to demonstrate non-smoothness of \(E_\mathsf{lo}(G)\) for certain groups $G$ of dynamical origin, such as certain  amalgams constructed from Thompson's group $F$.  We also initiate a systematic analysis of \(E_\mathsf{lo}(\pi_1(M))\), where $M$ is a $3$-manifold.  We prove that if $M$ is not prime, then \(E_\mathsf{lo}(\pi_1(M))\) is a universal countable Borel equivalence relation, and show that in certain cases the complexity of \(E_\mathsf{lo}(\pi_1(M))\) is bounded below by the complexity of the conjugacy equivalence relation arising from the fundamental group of each of the JSJ pieces of $M$.   We also prove that if $M$ is the complement of a nontrivial knot in $S^3$ then \(E_\mathsf{lo}(\pi_1(M))\) is not smooth, and show how determining smoothness of \(E_\mathsf{lo}(\pi_1(M))\) for all knot manifolds $M$ is related to the L-space conjecture.

\end{abstract}

%%%%%%%%%%----------Section 1----------%%%%%%%%%%

\section{Introduction}

The theory of ordered groups dates back as far as H\"older and Dedekind, however it is only over the last few decades that connections with orderable groups have brought to light a fruitful interplay between algebra,  topology, and dynamics. Recent advances in the theory of left-orderable groups crucially use this interaction. For example, Morris~\cite{Mor06} uses the dynamical properties of amenable groups to prove that every left-orderable amenable group is locally indicable, answering a question of Linnell~\cite{Lin99}.

One of the key tools in applying topological arguments to orderable groups is the compact space \(\LO(G)\) of all left-orderings of $G$, which comes equipped with a natural $G$-action by conjugation.  When $G$ is left-orderable, either $|\LO(G)| = 2^n$ and $G$ is a \emph{Tararin group},  or \(\LO(G)\) has the cardinality of the continuum \cite{Lin11}.  Since Tararin groups and their spaces of orderings are completely understood,  the structure of uncountable spaces of left-orderings is the remaining challenge facing the field. 
An emerging program in this direction is the analysis of uncountable spaces of left-orderings of countable groups by using the tools of descriptive set theory.  The goal is to classify left-orderable groups and their corresponding spaces of left-orderings in a complexity hierarchy using a finer notion of cardinality, namely \emph{Borel cardinality}.

The motivating question for this program was posed in the manuscript of Deroin, Navas, and Rivas~\cite{DerNavRiv}, which asks whether there exists a countable left-orderable group with non-standard quotient space \(\LO(G)/G\).  If the answer to this question were negative, then
recovering algebraic properties of a left-orderable group from the Borel structure of its orbit space would be impossible, as a famous theorem of Kuratowswki states that all uncountable standard Borel spaces are Borel isomorphic.

Fortunately this is not the case.  The main results of Calderoni and Clay~\cite{CalCla22} show that for many left-orderable groups the Borel structure of \(\LO(G)/G\) is more complicated than that of standard Borel spaces.  For example, if \(\LO(G)/G\) is standard  then \(G\) must be locally indicable; thus for any left-orderable group \(G\) that does not admit a Conradian left-order, the space \(\LO(G)/ G\) is not standard. This result fits more broadly in the topic of set theoretic rigidity, where the main question of study is how much information about the given group action is encoded by the Borel cardinality of the corresponding orbit space.

Our methods use the theory of countable Borel equivalence relations, which is of fundamental significance to the study of countable group actions and their corresponding orbit spaces.
The orbit space of a \(G\)-action is standard precisely when the corresponding orbit equivalence relation is smooth.
In particular, \cite{CalCla22} shows that a program to analyze left-orderable groups via the Borel complexity of the corresponding conjugacy actions is possible.

For a countable left-orderable group \(G\) denote by \(E_\mathsf{lo}(G)\) the countable Borel equivalence relation induced by the conjugacy \(G\)-action on \(\LO(G)\).
This article continues the study started in \cite{CalCla22,CalSha}, and determines new classes of countable left-orderable groups \(G\) for which \(E_\mathsf{lo}(G)\) is not smooth, and a new class for which \(E_\mathsf{lo}(G)\) is universal:

\begin{theorem}
\label{universal free products}
    If $G$ and $H$ are countable left-orderable groups, then \(E_\mathsf{lo}(G\ast H)\) is a universal countable Borel equivalence relation.
\end{theorem}

Our refined techniques for determining non-smoothness of \(E_\mathsf{lo}(G)\) rely upon either being able to control the normal subgroups of $G$, or upon being able to construct families of left-orderings of $G$ having prescribed algebraic properties.  Each of these techniques lends itself to an analysis of families of groups arising naturally from recent advances in the study of left-orderable groups.  We show:

\begin{theorem}
\label{simplenotbo}
If $G$ is a countable left-orderable group that admits no finite quotient and is not bi-orderable,  then \(E_\mathsf{lo}(G)\) is not smooth.
\end{theorem}

There is an abundant supply of groups satisfying these properties, arising from recent work in dynamics.  We will show how to use \cite{KimKob20} to create such groups in Section~\ref{sec : simple groups}. 

We also systematically analyze the complexity of \(E_\mathsf{lo}(G)\) whenever $G$ is a left-orderable fundamental groups of a \(3\)-manifold, using both Theorem \ref{universal free products} and new non-smoothness results that rely on constructing closed $G$-invariant subsets of $\LO(G)$. 

Our analysis begins by leveraging the prime decomposition of $M$. We first observe that for a compact, connected \(3\)-manifold \(M\), if \(M\) is not prime, then $E_\mathsf{lo}(\pi_1(M))$ is universal; whereas
if $M$ is prime and reducible, then $E_\mathsf{lo}(\pi_1(M))$ is smooth.  We also show how to use the JSJ decomposition in an analysis of $E_\mathsf{lo}(\pi_1(M))$, showing that it is universal in certain cases involving two JSJ pieces.

Then, for irreducible prime manifolds, we discuss a conjectural connection between
the Borel complexity of \(E_\mathsf{lo}(\pi_1(M))\) and the Heegaard Floer homology of \(M\), via the so called L-space conjecture.  This discussion shows how a combination of tools developed in \cite{BGH} and \cite{BC17} to study the L-space conjecture can be used alongside our non-smoothness results.  We show:

\begin{theorem}
\label{knot_result}
   If $G$ is a non-cyclic knot group, then \(E_\mathsf{lo}(G)\) is not smooth.
\end{theorem}

However we expect that a similar result holds for most knot manifolds, in particular we conjecture that if $M$ is a knot manifold and $E_\mathsf{lo}(\pi_1(M))$ is smooth, then $M$ is a generalized solid torus in the sense of \cite{RR17}.

The manuscript is organized as follows.  In Section \ref{sec: prelims} we cover background material on left-orderable groups and countable Borel equivalence relations, and in Section \ref{sec : free pruduct} we prove Theorem \ref{universal free products}. In  Section \ref{sec: new tricks} we use generic ergodicity to strengthen existing criteria for non-smoothness of \(E_\mathsf{lo}(G)\), prove Theorem \ref{simplenotbo} and discuss applications.  In Section~\ref{sec : 3-manifolds} we address the complexity of \(E_\mathsf{lo}(G)\) for left-orderable fundamental groups \(G\) of \(3\)-manifolds. In Section~\ref{sec : ses} we provide a generalized method of defining Borel reductions from short exact sequences of left-orderable groups, as such constructions are ubiquitous in our work.

\section{Preliminaries}
\label{sec: prelims}
\subsection{Left-orderable groups}

A group \(G\) is \emph{left-orderable} if it admits a strict total ordering \(<\) such that \(g < h\) implies \(fg < fh\) for all \(f, g, h \in G\).

\begin{proposition}
\label{prop : equivalent defs LO}
The following are equivalent:
\begin{enumerate}
\item
\(G\) is left-orderable.
\item
There is \(P\subseteq G\) such that
	\begin{enumerate}
	\item
	\(P \cdot P \subseteq P\);
	\item
	\(P\sqcup P^{-1} = G\setminus \{1_G\}\).
	\end{enumerate}
\item
There is a totally ordered set \((\Upomega, <)\) such that \(G\hookrightarrow Aut(\Upomega,<)\).
\end{enumerate}
\end{proposition}

The subset $P$ in $(2)$ above is referred to as a \emph{positive cone}.  It is clear that every left-ordering $<$ of $G$ determines a positive cone $P= \{ g \in G : g> 1\}$.  The identification of left-orderings with their positive cones allows us to define the space of left-orderings as follows.  Equip $\{0,1\}$ with the discrete topology, $\{0,1\}^G$ with the product topology, and set 
\[ \LO(G) = \{ P \subset G : P \mbox{ is a positive cone }\} \subset \{0,1\}^G,
\]
equipped with the subspace topology.  Thus the subbasic open sets in $\LO(G)$ are $U_g = \{ P : g \in P \}$, where $ g \in G\setminus \{ 1 \}$.  One can easily check that $\LO(G)$ is closed, hence compact.  There is a $G$-action by homeomorphisms on $\LO(G)$, given by $g \cdot P = gPg^{-1}$, which carries subbasic open sets to subbasic open sets.

The space $\LO(G)$ admits several distinguished $G$-invariant subspaces.  Recall that an ordering $<$ of a group $G$ is \emph{Conradian} if whenever $g, h \in G$ are both positive, then there exists $n \in \mathbb{N}$ such that $1<g^{-1}hg^n$ \cite{Con59}.  Equivalently, if $g, h \in G$ and $g, h >1$ then $1<g^{-1}hg^2$~\cite[Proposition 3.7]{Nav10}. 
We denote by $\CLO(G) \subseteq \LO(G)$ the (closed) subspace of all Conradian orders, one can check from the definition of Conradian orderings that $\CLO(G)$ is $G$-invariant. 

A group \(G\) is \emph{locally indicable} if for every finitely generated subgroup \(H<G\) there is an onto homomorphism \(h\colon H\to \mathbb{Z}\). As a consequence of the Burns-Hale theorem~(e.g., see~\cite[Theorem 1.50]{ClaRol}) every locally indicable group is left-orderable.  In fact, a group is locally indicable if and only if it admits a Conradian left-order \cite[Theorem 4.1]{Con59}.  

Now let \(G\) be a group equipped with a fixed left-ordering $<$. A subgroup $C$ of $G$ is \emph{convex relative to $<$} if whenever $g, h\in C$ and $f \in G$ with $g<f<h$, then $f \in C$.  A subgroup  \(C\subseteq G\) is \emph{left-relatively
convex} in \(G\) (or \emph{relatively convex} in \(G\) for short) if \(C\) is convex relative to some left ordering of \(G\).   When $C$ is normal in $G$, the quotient $G/C$ is left-orderable if and only if $C$ is left-relatively convex in $G$.

We require two results concerning the relative convexity of certain kinds of subgroups that will be needed in the later sections of this paper.  Recall that a subgroup \(H\) of \(G\) is \emph{isolated} if \(h^n \in H\) implies \(h \in H\). 

\begin{proposition}[{\cite[Lemma~3.2]{Cla12}}]
\label{prop : isolate lrc}
In a bi-orderable group \(G\),  every isolated
abelian subgroup is left-relatively convex.
\end{proposition}

\begin{proposition}[{\cite[Example~1.18]{ADZ18}}]
\label{prop : free lrc}
Suppose $A$ and $B$ are groups, $C \subset A$ a subgroup and $\phi: C \rightarrow B$ an injective homomorphism.  If $C$ and $\phi(C)$ are relatively convex in $A$ and $B$ respectively, then $A$ and $B$ are relatively convex in the free product with amalgamation $A*_{\phi} B$.
\end{proposition}

\subsection{Countable Borel equivalence relations}

Suppose that \(X\) is a set and that \(\mathcal{B}\) is a \(\sigma\)-algebra of subsets
of \(X\). Then \(( X, \mathcal{B})\) is a standard Borel space if there exists a Polish topology \(\tau\) on
\(X\) such that \(\mathcal{B}\) is the \(\sigma\)-algebra generated by \(\tau\).  An equivalence relation \(E\) on the standard Borel space \(X\) is said to be \emph{Borel} if \(E \subseteq X\times X\) is a Borel subset of \(X\times X\). A Borel equivalence relation \(E\) is said to be \emph{countable} if every \(E\)-equivalence class is countable. 

Most of the Borel equivalence relations that we will consider in this paper arise from group actions as follows. Let \(G\) be a countable group. Then a \emph{standard Borel \(G\)-space} is a standard Borel space \(X\) equipped with a Borel action \(G\times X\to X, (g, x) \mapsto g \cdot x\) of \(G\) on \(X\).
For any \(x\in X\), we denote by \(G\cdot x =\{g\cdot x : g\in G\}\) the orbit of \(x\). 
 We denote by \(E^X_G\) the \emph{orbit equivalence relation} on \(X\) whose classes are the \(G\)-orbits.  That is,
\[
x\mathbin{E}^X_Gy\quad \iff\quad G\cdot x = G\cdot y.
\]
\noindent
When \(G\) is a countable group,  then it is clear that \(E_G^X\) is a countable Borel equivalence relation.  Conversely, a theorem of Feldman and Moore~\cite{FelMoo}
states that
if \(E\) is an arbitrary countable Borel equivalence relation on some standard Borel space \(X\), then there exist a countable group \(G\) and a Borel action of \(G\) on \(X\) such that \(E = E_G^X\). 

When \(G\) is a countable left-orderable group, \(X= \LO(G)\), and we let \(G\) act on \(X\) by conjugacy,  we denote the corresponding orbit equivalence relation by \(E_\mathsf{lo}(G)\).  Moreover,  for any countable group \(G\) we denote by \(E_s(G)\) the countable Borel equivalence relation induced by the shift action of \(G\) on \(\{0,1\}^G\).
It is clear from the definition that both \(E_\mathsf{lo}\) and \(E_s(G)\) are examples of countable Borel equivalence relations.

Let \(E, F\) be countable Borel equivalence relations on the standard Borel spaces \(X, Y\) respectively. Then a Borel map \(f \colon X\to Y\) is said to be a \emph{homomorphism} from \(E\) to \(F\)  if for all \(x,y\in X\),
\[x\mathbin{E} y \implies f(x)\mathbin{F} f(y).\]
If \(f\) satisfies the stronger property that for all \(x, y \in X\),
\[x\mathbin{E}y \iff f(x)\mathbin{F}f(y),\]
then \(f\) is said to be a \emph{Borel reduction} from \(E\) to \(F\) and we write \(E \leq_B F\). 
Further we say that
\(E\) and \(F\) are \emph{Borel isomorphic} (in symbols, \(E\cong_B F\)) if
there is a Borel bijection \(f\colon X \to Y\) which is a Borel reduction from \(E\) to \(F\). 

Another way to look at Borel reducibility is to notice that \(E \leq_B F\) is equivalent to the existence of an injection from the quotient space \(X/E\) into \(Y /F\) which is “Borel”, in
the sense that it has a Borel lifting. Thus \(E \leq_B F\) is interpreted as saying
that \(X/E\) has \emph{Borel cardinality} less than or equal to that of \(Y /F\).
Moreover, in case \(E\leq_BF\leq_B E\) we say that \(X/E\) and \(Y/F\) have the same Borel cardinality, in symbols \(|X/E|_B =|X/F|_B \).

\begin{definition}
Let \(E\) be a Borel equivalence relation on the standard Borel space \(X\). We say that \(E\) is \emph{smooth} if there is a standard Borel space \(Y\) and a Borel function \(f\colon X \to Y\) such that \[ x_0\mathbin{E} x_1 \iff f(x_0) = (x_1).\]
\end{definition}

Note that since all uncountable standard Borel spaces are Borel isomorphic we can replace \(Y\) with \(\mathbb{R}\) in the definition of smooth equivalence relations.  Therefore when $X$ is uncountable, \(E\) is smooth if and only if \(|X/E|_B = |\mathbb{R}|_B\).

\begin{fact} 
\label{fact : smooth}
Suppose that \(E, F\) are countable Borel equivalence relations.
\begin{enumerate}
\item
\label{fact: 1}
If \(E\leq_B F\) and \(F\) is smooth, then \(E\) is smooth.
\item
\label{fact: 2}
Suppose that \(E\subseteq F\) and \(F\) is smooth, then \(E\) is smooth.
\end{enumerate}
\end{fact}

Notice that Fact~\ref{fact : smooth}\eqref{fact: 1} implies that
whenever \(X\) is a standard Borel \(G\)-space, and \(Y \subseteq X\) is Borel and \(G\)-invariant,
if \(E_G^X\) is smooth then \(E_G^Y\) is also a smooth equivalence relation.

Now we recall some useful characterizations of smoothness.
A \emph{transversal} for \(E\) is a set \(B\subseteq X\) which intersect each \(E\)-class in exactly one point.
A \emph{selector} for \(E\) is a function \(s \colon X \to X\)  such that \({s(x)}\mathbin{E} x\) and \(x\mathbin{E} y\) implies \(s(x) = {s(y)}\).

\begin{proposition}
\label{thm : characterization smooth}
Let \(E\) be a countable Borel equivalence relation on \(X\). The following are equivalent:
\begin{enumerate-(i)}
\item
\label{item : smooth1}
\(E\) is smooth;
\item
\label{item : smooth2}
\(E\) admits a Borel transversal;
\item
\label{item : smooth3}
\(E\) admits a Borel selector.
\end{enumerate-(i)}
\end{proposition}

\begin{comment}
\begin{proposition}
 Let \(E\) be a countable Borel equivalence relation. The following are equivalent:
\begin{enumerate}
\item \(E\) is smooth.

\item
There is a  function \(s\colon X \to X\) such that
\(s(x) \mathbin{E} x\) and \(x \mathbin{E} y \implies s(x) = s(y)\).
\item
\(E\) admits a Borel transversal. That is, there is a Borel \(T\subseteq X\)
which intersect each \(E\)-class in exactly one point.
\end{enumerate}
\end{proposition}
\end{comment}

The following proposition is a crucial consequence of generic ergodicity.
For a proof we refer the reader to \cite{CalCla22}.

	\begin{proposition}[{E.g. ~\cite[Proposition~2.4]{CalCla22}}]
	\label{prop : fin orbit}
	Suppose that \(G\) is a countable group acting by homeomorphisms on a compact Polish space \(X\) such that \(E_G^X\) is smooth. Then there exists a finite orbit.
	\end{proposition}

\section{Universality of free products}
\label{sec : free pruduct}

A countable Borel equivalence relation \(E\) is \emph{universal} if every
countable Borel equivalence relation \(F\) is Borel reducible to \(E\).
One of the primary tools from \cite{CalCla22} used to show universality is the following proposition.

\begin{proposition} \cite[Proposition 3.1]{CalCla22}
\label{convex lower bound}
If $C$ is left-relatively convex in $G$, and for all $h \in G$, $hCh^{-1} \subseteq C$ implies $h \in C$, then $E_\mathsf{lo}(C) \leq_B E_\mathsf{lo}(G)$.
\end{proposition}

We will use this, together with the fact that $E_\mathsf{lo}(\mathbb{F}_2)$ is universal \cite[Theorem 1.2]{CalCla22}, to prove the following.

\begin{proof}[Proof of Theorem \ref{universal free products}]
Vinogradov~\cite{Vin49} proved that the free product of left-orderable groups is left-orderable, so $\mathrm{LO}(G*H)$ is nonempty.
Consider the short exact sequence
\[ 1 \rightarrow K \rightarrow G*H \stackrel{\phi}{\rightarrow} G \times H \rightarrow 1,
\]
where $\phi$ is the homomorphism induced by the maps $\phi_G: G \rightarrow G \times H$ and $\phi_H :H \rightarrow G \times H$ given by $\phi_G(g) = (g, 1)$ and $\phi_H(h) = (1, h)$ for all $g \in G$ and $h \in H$.  For ease of exposition, set $x_{g,h} = [g,h] = ghg^{-1}h^{-1}$.  Then $K$ is a free group, with free basis $\{x_{g,h}\}$ where $g \in G \setminus \{1 \}$ and $h \in H \setminus \{1\}$ \cite[Proposition 4]{Ser}.  Observe that for all $g \in G \setminus \{1 \}$ and $h \in H \setminus \{1\}$,  if $b \in H$ then $bx_{g,h}b^{-1} = x^{-1}_{g,b}x^{}_{g,bh}$; and if $a \in G$, then $ax_{g,h}a^{-1} = x^{}_{ag, h}x_{a,h}^{-1}$.  Therefore 
\[
abx_{g,h}b^{-1}a^{-1} = x^{}_{a,b}x_{ag,b}^{-1}x^{}_{ag,bh}x_{a,bh}^{-1}.
\]
Fix nonidentity elements $s \in G$ and $t \in H$, and set $S = \{ x_{s,t}, x_{s^2, t^2}\}$ and choose $(a, b) \in G \times H$.  We first show that $abSb^{-1}a^{-1} \subset \langle \langle S \rangle \rangle_K$ implies $(a,b) = (1,1)$, here $ \langle \langle S \rangle \rangle_K$ is the normal closure in $K$ of $S$.  Note that elements of  $\langle \langle S \rangle \rangle_K$, when written as a reduced word in the generators $\{x_{g,h}\}$, must have zero exponent sum in every generator that does not lie in $S$.  So, we check that $abSb^{-1}a^{-1}$ always contains an element that, when expressed as a reduced word in $\{x_{g,h}\}$, contains a generator that is not in $S$ and which appears with nonzero exponent sum.

We consider cases.

\noindent \textbf{Case 1.} $a=1$.  Then $abx_{s,t}b^{-1}a^{-1} = x_{s,b}^{-1}x^{}_{s,bt}$, note that if $b \neq 1$ then $x_{s,b}$ and $x_{s,bt}$ are distinct, and $x_{s,bt} \notin S$.  Therefore in this case, $abSb^{-1}a^{-1} \subset \langle \langle S \rangle \rangle_K$ is only possible if $(a,b) = (1,1)$.

\noindent \textbf{Case 2.} $b=1$.  Then $abx_{s,t}b^{-1}a^{-1} = x^{}_{as,t}x_{a,t}^{-1}$. As in Case 1, if $a \neq 1$ then $x_{as,t}$ and $x_{a,t}$ are distinct, and $x_{as,t} \notin S$.  We conclude as in Case 1.

\noindent \textbf{Case 3.} $a \neq1$ and $b \neq 1$.  Then 
\[abx_{s,t}b^{-1}a^{-1} = x^{}_{a,b}x_{as,b}^{-1}x^{}_{as,bt}x_{a,bt}^{-1}
\]

\noindent \textbf{Subcase 3a.} $x_{a,b} \notin S$. Then observe that $x_{a,b}$ is distinct from $x_{as,b}$ and from $x_{a,bt}$, and so $x_{a,b}$ provides a generator that is not in $S$ and has nonzero exponent sum in $abx_{s,t}b^{-1}a^{-1}$, demonstrating $abSb^{-1}a^{-1} \not\subset \langle \langle S \rangle \rangle_K$.

\noindent \textbf{Subcase 3b.} $x_{a,b} \in S$.  In this case observe that $x_{as,b} \notin S$ and is distinct from $x_{a,b}$ and $x_{as,bt}$, and thus has nonzero exponent sum in $x_{as,b}$, demonstrating $abSb^{-1}a^{-1} \not\subset \langle \langle S \rangle \rangle_K$.

Now, set $C = \langle S \rangle \subset G*H$.  We verify that $C$ satisfies the required properties so that we may apply Proposition \ref{convex lower bound}.    First, we observe that $C$ is relatively convex in $G*H$, because $C$ is relatively convex in $K$ since it is a free factor of $K$, and $K$ is relatively convex in $G*H$ since $(G*H)/K \cong G \times H$ is left-orderable. 

Next, choose $w \in G*H$ and write $w=kab$ for some $k \in K$ and $a \in G$, $b \in H$.  First, observe that if $(a,b) \neq (1,1)$ then $abSb^{-1}a^{-1} \not\subset \langle \langle S \rangle \rangle_K$ and therefore $abCb^{-1}a^{-1} \not\subset \langle \langle C \rangle \rangle_K = \langle \langle S \rangle \rangle_K$.  Therefore $wCw^{-1} = kabCb^{-1}a^{-1}k^{-1} \not\subset \langle \langle C \rangle \rangle_K $, in particular, $wCw^{-1} \not \subset C$.

So we suppose that $w = k \in K$.  But then $wCw^{-1} \cap C \neq \{1\}$ implies $w \in C$, because $C$ is a free factor in $K$ and is therefore malnormal\footnote{Recall that
a subgroup \(H\) of  \(G\) is \emph{malnormal} if \(gHg^{-1} \cap H = \{1_G\}\) for all \(g \in G\)
with \(g \notin H\).} in $K$.

Thus $C$ satisfies the hypotheses of Proposition \ref{convex lower bound}, and as $C \cong \mathbb{F}_2$, we conclude that $E_\mathsf{lo}(\mathbb{F}_2) \leq_B E_\mathsf{lo}(G*H)$ and so $E_\mathsf{lo}(G*H)$ is universal.
\end{proof}

In contrast, if $G$ is a free product with amalgamation then $E_\mathsf{lo}(G)$ need not be universal --- in fact, it may be smooth.

\begin{example}
\label{Klein}
Consider the group $G=\langle a, b : a^2 = b^2 \rangle$, which is an amalgam of two copies of the integers. It is isomorphic to the  $K = \langle x, y : xyx^{-1} = y^{-1} \rangle$, which is a Tararin group having only four left-orderings. 
  Therefore  \(E_\mathsf{lo}(G)\) is smooth in this case.
\end{example}

\section{New sufficient condition(s) for non-smoothness of \texorpdfstring{$E_\mathsf{lo}(G)$}{EloG}}

\label{sec: new tricks}

The techniques of \cite{CalCla22} developed to show non-smoothness of  \(E_\mathsf{lo}(G)\) are targeted towards answering \cite[Question 2.2.11]{DerNavRiv}.  Here we develop several more general techniques that subsume those results, with the goal of determining when  \(E_\mathsf{lo}(G)\) is smooth for a given group $G$.  Our primary tool is Proposition~\ref{prop : fin orbit}.

Since the conjugacy action is trivial for abelian groups, it is immediate that abelian groups yield examples where  \(E_\mathsf{lo}(G)\) is smooth. However, \cite[Example 2.8]{CalCla22} provided a nonabelian example for which \(E_\mathsf{lo}(G)\) is smooth as well.  The next proposition generalizes that example.

\begin{proposition}
\label{prop : ab-by-finite}
Suppose that $G$ is countable, virtually abelian and left-orderable.  Then $E_\mathsf{lo}(G)$ is smooth.
\end{proposition}
\begin{proof} 
Suppose $A \leq G$ is abelian and $|G:A| =n$, we first show that the action of the subgroup $A$ on $\LO(G)$ is trivial. To this end, let $P, Q \in \LO(G)$ and suppose that $P \cap A = Q \cap A$.  Then $g \in P$ implies the existence of some \(1\leq k \leq n\) such that $g^k \in P \cap A$, so that $g^k \in Q$.  But then $g \in Q$, and so $P \subseteq Q$ and therefore $P =Q$.  Using this, we conclude that if $a \in A$ and $P \in \LO(G)$, then as $aPa^{-1} \cap A = a(P \cap A)a^{-1}$ we must have $aPa^{-1} = P$.

Now let $\{g_1, \ldots, g_n\}$ be a complete set of left coset representatives for $A$ in $G$.  Then given $g \in G$, write $g = g_k a$ for some $k \in \{1, \ldots ,n\}$ and $a \in A$, and note that for every $P \in \LO(G)$ we have $gPg^{-1} = g_kaPa^{-1}g_k^{-1} = g_kPg_k^{-1}$.  From this it follows that the orbit of $P$ is $\{g_1Pg_1^{-1}, \ldots, g_nPg_n^{-1}\}$.  As $g \in G$ and $P \in \LO(G)$ were arbitrary, we conclude that  every orbit is finite.

Since \(\LO(G)\) is a Polish space,  it admits a Borel ordering \(\prec\). Therefore
the map \(s\colon \LO(G)\to \LO(G), s(P)= \min_{\prec}\{Q\in\LO(G) : (Q,P)\in E_{\mathsf{lo}}(G)\}\) is a Borel selector for \(E_\mathsf{lo}(G)\) and so $E_\mathsf{lo}(G)$ is smooth. 
\end{proof}

\begin{remark}
To contrast Proposition~\ref{prop : ab-by-finite}, note that the analogue fails for abelian-by-cyclic groups.  Recall that a group \(G\) is \emph{abelian-by-cyclic} if there is a short exact sequence
\[
1 \rightarrow A \stackrel{}{\rightarrow} G \stackrel{}{\rightarrow} Z \rightarrow 1,
\]
where \(A\) is an abelian group and \(Z\) is an infinite cyclic group.  An example of an abelian-by-cyclic group is the (restricted) wreath product \(\mathbb{Z}\wr \mathbb{Z} = \mathbb{Z}^{(\mathbb{Z})}\rtimes \mathbb{Z}  \),
where \(\mathbb{Z}^{(\mathbb{Z})}=\{f\colon \mathbb{Z}\to \mathbb{Z} \mid \mathrm{supp}(f)\text{ is finite}\}\) and \(\mathbb{Z}\) acts on \(\mathbb{Z}^{(\mathbb{Z})}\) by shift. Since \(\mathbb{Z}\) is bi-orderable,  one can prove that \(E_s(\mathbb{Z}) \leq_B E_\mathsf{lo}(\mathbb{Z}\wr \mathbb{Z})\) arguing as in~\cite[Proposition~4.2]{CalCla22}.  Moreover, \(E_s(\mathbb{Z})\) is invariantly universal for hyperfinite Borel equivalence~\cite[Section~9]{DouJacKec} thus it is not smooth. The following question remains open.
\end{remark}

\begin{question}
\label{Question : hyp}
Is \(E_\mathsf{lo}(\mathbb{Z}\wr\mathbb{Z})\) hyperfinite?
\end{question}

Since \(\mathbb{Z}\wr\mathbb{Z}\) is amenable,
question~\ref{Question : hyp} is a particular instance of the most long-standing open questions in the theory of countable Borel equivalence relations: 

\begin{question}[{Weiss~\cite{Wei84}}]
    If \(G\) is a countable amenable group and \(X\) is a standard
Borel space \(G\)-space, must the orbit equivalence relation
\(E^X_G\) be hyperfinite?
\end{question}

\subsection{Restrictions on subgroups of \texorpdfstring{$G$}{G}}

\label{sec : simple groups}

Suppose that \(X\) is a standard Borel \(G\)-space. For any \(x\in X\),  let \(G_x\coloneqq\{g\in G : g\cdot x = x\}\) be the \emph{stabilizer} of \(x\).

\begin{proposition} 
\label{prop : finite-index subg}
If the orbit \(G\cdot x\) is finite and contains more than one element,  then there is a proper subgroup $N$ with \(N\leq G_x\) with \(N\lhd G\) and \([G: N]<\infty\). 
\end{proposition}

\begin{proof}
Suppose that \(|G\cdot x| = k <\infty\).  Since we can identify the set of left-cosets \(\{gG_x: g\in G\}\) with the finite orbit \(G\cdot x\),  it is clear that \(G_x\leq G\), with \([G:G_x]=k\). 

Then consider the action by left-translation of \(G\) on the set of left-cosets
\(\{gG_x: g\in G\}\). This induces a representation \(\rho\colon G\to S_k\). Set \(N=\ker \rho\).  It follows that \(G/N\) is isomorphic to a nontrivial subgroup of \(S_k\), which is finite.
\end{proof}

\begin{proposition} 
\label{cor : simple}
Suppose \(X\) is compact and \(G\curvearrowright X\) by homeomorphisms. If \(E^X_G\) is smooth and there is no $x \in X$ with $g \cdot x =x$ for all $g \in G$, then there must be a finite orbit and \(G\) has a nontrivial finite quotient.
\end{proposition}

\begin{proof}
Suppose \(E_G^X\) is smooth. By Proposition~\ref{prop : fin orbit} there must be a finite orbit \(G\cdot x \in X/G\), which by assumption contains more than one element. By letting \(N\leq G_x\) as in Proposition~\ref{prop : finite-index subg}, we have \(|G/N| <\infty\) with $G/N$ nontrivial.
\end{proof}

\begin{proof}[Proof of Theorem \ref{simplenotbo}]
Note $P \in \LO(G)$ is the positive cone of a bi-ordering if and only if \(gPg^{-1} = P\) for all \(g\in G\).  To see this, suppose that $<$ is the left-ordering of $G$ associated to $P$.  If $<$ is a bi-ordering, then clearly \(gPg^{-1} = P\) for all $g \in G$.  On the other hand, suppose \(gPg^{-1} = P\) for all $g \in G$ and $f<h$.  Then $f^{-1} h \in P$ and so $g^{-1}f^{-1}hg \in P$ for every $g \in G$, and therefore $fg<hg$.  
With this fact in hand, we apply Proposition~\ref{cor : simple}.
\end{proof}

Note that in particular, this result implies that if $G$ is simple and not bi-orderable, then \(E_\mathsf{lo}(G)\) is not smooth.

\subsection{Examples: Free products of simple groups with amalgamation}

An abundant supply of left-orderable groups with restrictions on the existence of normal subgroups arises from the recent discovery of finitely generated, left-orderable simple groups~\cite{HydLod}, as well as the investigation of critical regularity and associated smoothness questions concerning actions of left-orderable groups on the real line~\cite{KimKob20}.

Since a finitely generated, left-orderable simple group $G$ can never be locally indicable, the results of \cite{CalCla22} already show that $E_\mathsf{lo}(G)$ is not smooth in this case.  However these results say nothing about the case of countable left-orderable groups admitting no finite quotients which are locally indicable, plenty of which are known to exist.   Below, we construct such a class of groups, and apply Theorem~\ref{simplenotbo}. 

\begin{lemma}
\label{no finite quotients amalgam}
Suppose that $G$ and $H$ are simple groups, and $C$ an infinite cyclic group equipped with injective homomorphisms $\phi_G\colon C \rightarrow G$ and $\phi_H\colon C \rightarrow H$.  Then the free product $G*_CH$ amalgamated via the maps $\phi_G, \phi_H$ admits no finite quotient.
\end{lemma}
\begin{proof}
Suppose that $N$ is a finite index normal subgroup of $G*_CH$.  As $C$ is infinite cyclic, $N \cap \phi_G(C)$ is infinite, as is $N \cap \phi_H(C)$.  But then $N \cap G$ is a nontrivial normal subgroup of $G$, and since $G$ is simple, this implies $G \subset N$.  Similarly we conclude that $H \subset N$, so that $N = G*_CH$.
\end{proof}

We also require the following well-known lemma.

\begin{lemma}
\label{not BO lemma}
Suppose that $G$ is a bi-orderable group, and that $g, h \in G$ and $[g^p, h^q] = 1$ for some $p, q \in \mathbb{Z} \setminus \{0 \}$.   Then $[g,h] =1$.
\end{lemma}
\begin{proof}
Suppose that $g, h \in G$ with $[g^p, h^q] = 1$, and that $[g,h] \neq 1$.  Then $ghg^{-1} \neq h$, we may assume that $ghg^{-1} < h$ for some bi-ordering of $G$.  Assume $p, q>0$, the other cases being similar.  Then $g^p h g^{-p} < g^{p-1} h g^{1-p} < \dots < h$ and therefore $g^p h^q g^{-p} < h^q$, so that $[g^p, h^q] <1$, a contradiction.
\end{proof}

\begin{proposition}
\label{nonbo examples}
Suppose that $G$, $H$ are locally indicable simple groups, and $C = \langle t \rangle $ is an infinite cyclic group.  Choose nonidentity elements $g \in G$ and $h \in H$ and integers $p, q \in \mathbb{Z} \setminus \{0, \pm 1\}$, and define $\phi_G \colon C \rightarrow G$ and $\phi_H \colon C \rightarrow H$ by $\phi_G(t) = g^p$ and $\phi_H(t) = h^q$.  Then $G*_CH$ is locally indicable, not bi-orderable, and admits no finite quotient.
\end{proposition}
\begin{proof}
That $G*_CH$ admits no finite quotient follows from Lemma \ref{no finite quotients amalgam}. To see that $G$ is not bi-orderable, note that our choice of $p$ and $q$ implies $g \notin \phi_G(C)$ and $h \notin \phi_H(C)$, and therefore $g$ and $h$ do not commute.  However $\phi_H(t) = h^q$ and $\phi_G(t) = g^p$ implies $[g^p, h^q] =1$, so that $G*_CH$ cannot be bi-orderable by Lemma~\ref{not BO lemma}.  It is locally indicable by \cite[Theorem 9]{KS70}.
\end{proof}

Proposition \ref{nonbo examples} therefore yields plenty of examples where we can apply  Corollary \ref{simplenotbo}, provided we have an ample supply of locally indicable, countable simple groups.  As a first example, we consider Thompson's group $F$ and its commutator subgroup.

\begin{example}
Recall that Thompson's group $F$ has presentation 
\[ F = \langle a, b \mid [ab^{-1}, a^{-1}ba] = [ab^{-1}, a^{-2}ba^2] = id \rangle.
\]
It is well-known that this group is bi-orderable, hence it is also locally indicable~\cite{NavRiv}.  The commutator subgroup $F'$ is simple, by an application of the main result of~\cite{Higman54}.  Suppose $C = \langle t \rangle$ is an infinite cyclic group, and define $\phi_1, \phi_2 : C \rightarrow F'$ by $\phi_1(t) = g^p, \phi_2(t) = h^q$, then set $K = F'*_C F'$, with amalgamation via the maps $\phi_i$.  Then $E_\mathsf{lo}(K)$ is not smooth by Proposition~\ref{nonbo examples}.
\end{example}

We can produce many more examples using the dynamical results of \cite{KimKob20} and the group $F$. We first require two preliminary ingredients.  

Recall that Thompson's group $F$ can be realized as a subgroup of $\mathrm{Homeo}_+([0,1])$, namely the subgroup of piecewise-linear maps whose breakpoints have dyadic rational coordinates. This yields a embedding of $\rho\colon F \rightarrow \mathrm{Homeo}_+([0,1])$, which by Ghys-Sergiescu~\cite{GS87} is topologically conjugate to an embedding $\rho_{GS}\colon F \rightarrow \Diff_+^{\infty}([0,1])$.

We also recall that a homeomorphism $f\colon[0,1] \rightarrow [0,1]$ is \emph{piecewise differentiable of class $C^1$}, or \emph{piecewise $C^1$} for short, if there exists a finite subset $S_f \subset [0,1]$ such that $f$ is a $C^1$ diffeomorphism on $ [0,1] \setminus S_f $, and at every point $x \in S_f$ the left and right derivatives of $f$ exist.  The group of all piecewise $C^1$ diffeomorphisms of $[0,1]$ is denoted $\mathrm{P}\Diff^1([0,1])$.  Given a subgroup $H$ of $\mathrm{P}\Diff^1([0,1])$, the \emph{support of $H$}, denoted $\mathrm{supp}(H)$, is the closure of the set
\[\{ x \in [0,1] \mid \exists h \in H \mbox{ such that } h(x) \neq x \}.
\]

The following lemma is well known to experts.

\begin{lemma}
\label{piecewise thurston}
   The group $\mathrm{P}\Diff^1([0,1])$ is locally indicable. 
\end{lemma}
\begin{proof}
  Suppose that $H \subset   \mathrm{P}\Diff^1([0,1])$ is a finitely generated subgroup, say with generators $\{ h_1, \dots, h_n\}$ and set $p = \inf \mathrm{supp}(H)$.  Since every element of $H$ fixes $p$, there is an inclusion $\phi\colon H \rightarrow \mathrm{Fix}(p) \subset\mathrm{P}\Diff^1([0,1])$. 
  Now the group of $C^1$ germs fixing $p$ can be identified with $G_p = \mathrm{Fix}(p)/N_p$, where 
  \[ N_p = \{ f \in \mathrm{Fix}(p) \mid \exists U \mbox{ open with }p \in U \mbox{ and } f|_U = id \}.
  \]
  So there is a homomorphism $H \rightarrow G_p$ by following $\phi$ with the quotient map $\mathrm{Fix}(p) \rightarrow G_p$, and since $G_p$ is locally indicable by \cite[Theorem 2.119]{Calegari2007}, we need only to check that the image of this homomorphism is nontrivial.

  To this end, suppose that $\phi(H) \subset N_p$, and for each $h_i $ choose an open set $U_i$ with $p \in U_i$ and $h_i |_{U_i} = id.$  Set $U = \bigcap_{i=1}^n U_i$, then $h_i|_U = id$ for all $i$ and $p \in U$, which contradicts $p =\inf \mathrm{supp}(H)$.  This finishes the proof.
\end{proof}

\begin{proposition}
Suppose that $G \subset \mathrm{P}\Diff^1([0,1])$ is a countable subgroup whose support is compactly contained in $[0,1]$, and let $\tilde{G}$ denote the subgroup of $\mathrm{P}\Diff^1([0,1])$ generated by $G \cup \rho_{GS}(F)$. Then $[\tilde{G}, \tilde{G}]$ is simple, countable and locally indicable.
\end{proposition}
\begin{proof}
That $[\tilde{G}, \tilde{G}]$  is simple is a consequence of \cite[Lemma 6.4]{KimKob20}, while local indicability follows from Lemma~\ref{piecewise thurston}.
\end{proof}

\begin{question}
    Are there simple, locally indicable groups that are not bi-orderable?
\end{question}

 \subsection{Restrictions on orderings of \texorpdfstring{$G$}{G}}

 One of the key tools from \cite{CalCla22} was the observation that if $P \in \LO(G)$ has a finite orbit, then in fact $P \in \CLO(G)$, the subspace of all Conradian orderings of $G$ (see \cite[Proof of Theorem 1.1]{CalCla22}).  This observation was used to prove that if $G$ is not Conradian left-orderable (i.e. if $\CLO(G)$ is empty), then \(E_\mathsf{lo}(G)\) is not smooth because there cannot be a finite orbit. We can generalize this line of reasoning, arriving at the following:

\begin{proposition}
\label{prop : LO minus CO}
If there exists $X \subset \LO(G) \setminus \CLO(G)$ that is both closed and $G$-invariant, then $E_\mathsf{lo}(G)$ is not smooth.
\end{proposition}
\begin{proof}
If $E_\mathsf{lo}(G)$ is smooth and $X \subset \LO(G)$ is $G$-invariant, then $E_G^X$ is smooth.  Thus the action of $G$ on $X$ has a finite orbit by Proposition \ref{prop : fin orbit}.  However this is not possible since all finite orbits lie in $\CLO(G)$, while $X \subset \LO(G) \setminus \CLO(G)$.
\end{proof}

There are natural techniques for producing closed, $G$-invariant subsets $X \subset \LO(G) \setminus \CLO(G)$ as in Proposition \ref{prop : LO minus CO}.  For instance, we can construct such a set of orderings via short exact sequences as follows:

\begin{proposition}
\label{quotient condition}
Suppose that 
\[1 \rightarrow K \stackrel{i}{\rightarrow} G \stackrel{q}{\rightarrow} H \rightarrow 1
\]
is a short exact sequence of groups, and that $G$ and $H$ are left-orderable.   If there exists a closed, $H$-invariant set $X \subset \LO(H) \setminus \CLO(H)$ then $E_\mathsf{lo}(G)$ is not smooth.
\end{proposition}
\begin{proof}
Define a map $\lex\colon{\LO(K)} \times {\LO(H)} \rightarrow \LO(G)$ by setting \[\lex(P_K, P_H) = i(P_K) \cup q^{-1}(P_H).\] Then $\lex$ is a continuous map, which can be checked by verifying that the preimage of a subbasic open set $U^G_g = \{ P \in \LO(G) \mid g \in P\}$ is either $U^K_{i^{-1}(g)} \times \LO(H)$ or $\LO(K) \times U^H_{q(g)}$ depending on whether $g \in i(K)$ or $q(g) \neq 1$.  Here, the superscripts on the subbasic open sets indicate whether we are working with subbasic open sets of $\LO(K), \LO(H)$ or $\LO(G)$.

Set $\lex(\LO(K) \times X) = Y$, and note that $Y$ is a compact set.  We claim that $Y \subset \LO(G) \setminus \CLO(G)$ and that $Y$ is $G$-invariant.

We first show that $\lex(P_K, P_H) \in \CLO(G)$ if and only if both $P_K \in \CLO(G)$ and $P_H \in \CLO(H)$.  To this end, suppose that $\lex(P_K, P_H) = i(P_K) \cup q^{-1}(P_H) \notin \CLO(G)$.  Then there exist $g, h \in lex(P_K, P_H)$ such that $g^{-1}hg^2 \notin \lex(P_K, P_H)$.  There are three possibilities:

\noindent \textbf{Case 1.}  $g \in i(P_K)$ and $h \in q^{-1}(P_H)$.  In this case, $q(g^{-1} h g^2) = q(h) \in P_H$, so that $g^{-1}hg^2 \in \lex(P_K, P_H)$  meaning this case is impossible.  Similarly the case of $h \in i(P_K)$ and $g \in q^{-1}(P_H)$ is impossible, as then $q(g^{-1} h g^2) = q(g) \in P_H$.

\noindent \textbf{Case 2.} $g, h \in i(P_K)$.  Then $g^{-1}hg^2 \notin \lex(P_K, P_H)$ implies $g^{-1}hg^2 \notin i(P_K)$, so that $P_K$ is not the positive cone of a Conradian ordering. 

\noindent \textbf{Case 3.} $g, h \in q^{-1}(P_H)$.   Then $g^{-1}hg^2 \notin \lex(P_K, P_H)$ implies $g^{-1}hg^2 \notin q^{-1}(P_H)$, so that $q(g), q(h) \in P_H$ while $q(g)^{-1}q(h)q(g)^2 \notin P_H$, so that $P_H$ is not the positive cone of a Conradian ordering.

On the other hand, if $P_K$ or $P_H$ is not the positive cone of a Conradian ordering, then it is clear that $\lex(P_K, P_H) \notin \CLO(G)$ via reasoning similar to above.  It follows that $\lex(\LO(K) \times X) = Y \cap \CLO(G) = \emptyset$.

Last we prove that $Y$ is $G$-invariant.  Given $g \in G$ and $P_K \in \LO(K)$, $P_H \in X$, then 
\[g^{-1}\lex(P_K, P_H)g = g^{-1}i(P_K)g \cup g^{-1}q^{-1}(P_H)g.
\]
Note that $g^{-1}i(P_K)g \in \LO(K)$ and that $g^{-1}q^{-1}(P_H)g = q^{-1}(q(g)^{-1}P_Hq(g))$, where $q(g)^{-1}P_Hq(g) \in X$ since $X$ is $H$-invariant.  Thus 
\[
g^{-1}\lex(P_K, P_H)g = \lex(g^{-1}i(P_K)g, q(g)^{-1}P_Hq(g)) \in Y.
\]
This shows that \(Y\) is $G$-invariant, and completes the proof as the statement now follows from Proposition~\ref{prop : LO minus CO}.
\end{proof}

\begin{corollary}
\label{quotient corollary}
Suppose that $G$ is left-orderable and admits a left-orderable quotient that is not locally indicable.  Then $E_\mathsf{lo}(G)$ is not smooth.
\end{corollary}
\begin{proof}
Denote the left-orderable non-locally indicable quotient by $H$, and apply the Proposition \ref{quotient condition} with $X = \LO(H)$.
\end{proof}

\begin{remark}
Note that when \(G\) is bi-orderable, we can conclude that \(E_\mathsf{lo}(H)\leq_BE_\mathsf{lo}(G)\) by Corollary~\cite[Corollary~3.5]{CalCla22}. 
\end{remark}

\section{The L-space conjecture and \texorpdfstring{$3$}{3}-manifold groups}
\label{sec : 3-manifolds}
Let $M$ be a compact, orientable $3$-manifold.  Over the last decade, orderability of $\pi_1(M)$ has become an active area of study in low-dimensional topology, as this group may be left-orderable (for instance, every knot group is left-orderable since it has infinite cyclic abelianization \cite[Theorem 1.1]{BRW05}) while in other cases the group is not left-orderable, such as when $M$ is a lens space, the Hantzsche-Wendt manifold \cite[Example 1.59]{ClaRol}, or the Weeks manifold \cite[Example 5.11]{ClaRol}.   Whether or not $\pi_1(M)$ is left-orderable is conjecturally related to the Heegaard Floer homology of $M$, and whether or not $M$ supports a co-orientable taut foliation --- a conjecture known as the L-space conjecture (see Conjecture~\ref{Lspace conjecture} below).

The goal of this subsection is to explain how smoothness of $E_\mathsf{lo}(\pi_1(M))$ when $\pi_1(M)$ is left-orderable is connected to the prime and JSJ decompositions of $M$, and observe some of the expected behaviors that are a consequence of the L-space conjecture. 

Recall that a $3$-manifold $M$ is \emph{prime} if \(M\cong M_1 \mathbin{\#}M_2\) implies \(M_1\cong S^3\) or \(M_2\cong S^3\), where $\#$ denotes connect sum.
Every $3$-manifold $M$ admits a \emph{prime decomposition}, that is, it can be expressed uniquely as a connect sum of prime $3$-manifolds 
\[ M \cong M_1 \mathbin{\#} \ldots \mathbin{\#} M_n,
\]
from which we conclude $\pi_1(M) \cong \pi_1(M_1) * \ldots * \pi_1(M_n)$.  

Closely related is the notion of irreducibility:  A $3$-manifold $M$ is \emph{irreducible} if every embedded $2$-sphere in $M$ bounds a $3$-ball.  There are only two $3$-manifolds which are prime but not irreducible, either $S^2 \times S^1$ or the twisted $2$-sphere bundle over $S^1$.  From this we observe that the fundamental group of a prime, reducible manifold is always $\mathbb{Z}$, and so the prime decomposition of a $3$-manifold $M$ above yields 
 \[
 \pi_1(M) \cong \pi_1(M_1) \ast \ldots \ast \pi_1(M_k) \ast \mathbb{Z} \ast \ldots \ast \mathbb{Z}
  \]
where $n-k$ is the number of prime, reducible factors (corresponding to the copies of $\mathbb{Z}$), and $M_1, \ldots, M_k$ are irreducible.  Since the free product of left-orderable groups is left-orderable~\cite{Vin49}, it follows that $\pi_1(M)$ is left-orderable if and only if $\pi_1(M_i)$ is left-orderable for each $i = 1, \ldots, k$.  From this we observe:

\begin{proposition}Suppose that $M$ is a compact, connected orientable $3$-manifold whose fundamental group is left-orderable.  If $M$ is not prime,  then  $E_\mathsf{lo}(\pi_1(M))$ is universal; and if $M$ is prime and reducible, then   $E_\mathsf{lo}(\pi_1(M))$ is smooth.
\end{proposition}
\begin{proof}
This follows from Theorem \ref{universal free products} and the observations above.
\end{proof}

We therefore focus on investigating $E_\mathsf{lo}(\pi_1(M))$ in the case that $M$ is an irreducible $3$-manifold.

To do this, we employ a further unique decomposition of $M$ into pieces, namely the JSJ decomposition.  Before introducing this decomposition, we introduce some standard terminology from $3$-manifold topology, following \cite{Hatcher}.   A two-sided surface $S$ in $M$ is \emph{incompressible} if, for each disk $D$ in $M$ with $D \cap S = \partial D$, there is a disk $D' \subset S$ with $\partial D' = \partial D$.  A surface $S$ is called \emph{$\partial$-parallel} if there is an isotopy fixing $\partial S$ carrying $S$ into $\partial M$.  A $3$-manifold $M$ is \emph{atoroidal} if every incompressible torus in $M$ is $\partial$-parallel. A Seifert fibered $3$-manifold is a $3$-manifold which admits a foliation by circles \cite{Eps72}.

The JSJ decomposition is a decomposition of $M$ into pieces as follows: If $M$ is compact, connect, orientable, irreducible
 and the boundary of $M$ (if nonempty) consists of a union of tori, then there exists a unique (up to isotopy) collection $\mathcal{T}$ of disjoint, incompressible embedded tori such that
\[
M \setminus \mathcal{T} = M_1 \cup \dotsb \cup M_n,
\] where each $M_i$ is either Seifert fibered or atoroidal.   This reduces the problem of analyzing $E_\mathsf{lo}(\pi_1(M))$ to the problem of analyzing fundamental groups of graphs of groups with $Z \oplus Z$ edge groups (corresponding to the tori of $\mathcal{T}$), and vertex groups given by the fundamental group of atoroidal $3$-manifolds and Seifert fibred manifolds whose boundary is a union of tori (when $M$ consists of many pieces), or empty, if $M$ consists of a single piece (for background on graphs of groups, see~\cite{Ser}).  

In general, if $G$ is a fundamental group of a graph of groups, then relating $E_\mathsf{lo}(G)$ to the vertex and edge groups is difficult.  However there are sufficient tools available in $3$-manifold topology that we can show:

\begin{theorem}
\label{JSJ example}
Suppose that for $i = 1, 2$, the $3$-manifold $M_i$ is compact, connected, orientable and irreducible, and that $M_i$ is not homeomorphic to $S^1 \times S^1 \times [0,1]$, and that the boundary of each $M_i$ consists of a union of incompressible tori.  Suppose that there exists a choice of boundary torus $T_i \subset \partial M_i$ such that the JSJ piece of $M_i$ containing $T_i$ is not Seifert fibred.
 
Fix a homeomorphism $\phi \colon T_1 \rightarrow T_2$ and set $M = M_1 \cup_{\phi} M_2$.  If $\pi_1(M_i)$ are bi-orderable for $i=1, 2$, then $E_\mathsf{lo}(\pi_1(M_i)) \leq_B E_\mathsf{lo}(\pi_1(M))$ for $i = 1, 2$.
\end{theorem}

Before proceeding with the proof of this theorem, we require the following well-known proposition, whose proof we include for completeness. 

\begin{proposition}
\label{malnormal amalgam}
In an amalgam \(G=A\ast_CB\),  \(A\) is malnormal in \(G\) if and only if \(C\) is malnormal in \(B\).
\end{proposition}
\begin{proof}
First suppose that \(C\) is malnormal in \(B\).
Let \(a\in A\), \(a\neq 1\), and \(x_1x_2\dotsm x_m \in G\setminus A\),  written so that \(m\geq 1\) and \(x_1,x_2,\dotsc,x_m\) are alternately from \(A\setminus C\),   and \(B\setminus C\).  It is a standard result that an alternating product of this form lies in one of the factors $A$ or $B$ if and only if $m=1$. We want to prove that 
\(w=x_m^{-1}\dotsm x_1^{-1} a x_1\dotsm x_m \notin A\).
We consider two cases:
\begin{description}
\item [Case~1]
	\(x_1\in B\setminus C\).   Then  either \(w\) is already an alternating product with more than one term if \(a\in A\setminus C\), so \(w\) is not in \(A\); or  \(w=x_m^{-1}\dotsm x_2^{-1}(x_1^{-1}ax_1)x_2\dotsm x_m\) is an alternating product if \(a\in C\) with either a single term lying in $B\setminus C$ or more than one term.  In either case, \(w\) is not in \(A\). 

\item [Case~2]
\(x_1\in A\setminus C.\) Then \(m\geq 2\). We see that either \(t\coloneqq x_1^{-1}ax_1\) is in \(C\) or in \(A\setminus C\), and proceed as in Case 1.  If \(t\) is in \(C\), then \(w=x_m^{-1}\dotsm x_3^{-1}(x_2^{-1}tx_2) x_3\dotsm x_m\) is either a single term lying in $B\setminus C$, or an alternating product with more than one term. If \(t\) is in \(A\setminus C\),  then \(w=x_m^{-1}\dotsm x_3^{-1}x_2^{-1}t x_2x_3\dotsm x_m\) is an alternating product with more than one term. In either case, \(w\) is not in \(A\). 
\end{description}

If \(C\) is not malnormal in \(B\), then there are some nontrivial \(c\in C\) and  \(b\in B\setminus C\),  such that \(bcb^{-1}\in C\). Then this \(c\) is nontrivial in \(A\) and \(b\) is in \(G\setminus A\) with \(bcb^{-1}\in A\).  Hence \(A\) is not malnormal in \(G\). 
\end{proof}

\begin{proof}[Proof of Theorem \ref{JSJ example}]
Since the boundary tori of $M_i$ are incompressible, the inclusion map $T_i \rightarrow M_i$ induces an injective homomorphism $\pi_1(T_i) \rightarrow \pi_1(M_i)$. Therefore $\pi_1(M) \cong \pi_1(M_1) \ast_{\phi_*} \pi_1(M_2),$ the amalgam of $\pi_1(M_1)$ and $\pi_1(M_2)$ with respect to the isomorphism $\phi_* \colon \pi_1(T_1) \rightarrow \pi_1(T_2)$, by the Seifert-Van Kampen theorem.   Moreover, for each $i$, the manifold $M_i$ is not $ S^1 \times D^2$ since the boundary of $M_i$ is incompressible, so the group $\pi_1(T_i)$ is malnormal in each $\pi_1(M_i)$ by \cite[Theorem 3]{HW14}.  Therefore the subgroup $\pi_1(M_i)$ is malnormal in $\pi_1(M)$ by Proposition \ref{malnormal amalgam}.

Next, suppose that $g \in \pi_1(M_i) \setminus \pi_1(T_i)$, and that $g^k \in \pi_1(T_i)$ for some $k>1$.   Then $\langle g^k \rangle \subset g \pi_1(T_i)g^{-1} \cap \pi_1(T_i)$, which contradicts malnormality of $\pi_1(T_i)$.  So it must be that $g^k \in \pi_1(T_i)$ for some $k>1$ implies $g \in \pi_1(T_i)$, meaning that the subgroup $\pi_1(T_i)$ is relatively convex by Proposition~\ref{prop : isolate lrc}.  It follows that for each $i =1, 2$, the subgroup $\pi_1(T_i)$ is relatively convex in $\pi_1(M_i)$ by Proposition~\ref{prop : free lrc}. 
The conclusion now follows from Proposition~\ref{convex lower bound}.
\end{proof}

This shows that in some cases (and we expect in much more generality), the issue of smoothness of $E_\mathsf{lo}(\pi_1(M))$ reduces to considering the pieces of the JSJ decomposition of $M$, that is, irreducible $3$-manifolds with nonempty boundary consisting of a collection of disjoint incompressible tori.  

From here forward, we focus on an analysis of the simplest possible kind of JSJ piece, called a \emph{knot manifold}, which is a compact, connected, orientable and irreducible manifold $M$ other than $S^1 \times D^2$ having boundary a single torus $T$ and $H_1(M; \mathbb{Q}) \cong \mathbb{Q}$.  For such manifolds, $\pi_1(M)$ is always left-orderable, since $|H_1(M; \mathbb{Z})| = \infty$ and so $\pi_1(M)$ admits a surjection onto $\mathbb{Z}$ \cite[Theorem 1.1]{BRW05}.  

Recall that an element $g \in G$ is called \emph{primitive} if $g = h^n$ for some $h \in G$, $n>0$, implies $g=h$ and $n=1$.  A primitive element
of $\pi_1(T)$ whose image has finite order in $H_1(M;\mathbb{Z})$ will be called a \emph{rational longitude}, we fix a choice of such an element and denote it by $\lambda_M$.  Such an element $\lambda_M$ always exists, by a standard rank argument based on the long exact sequence in homology arising from the pair $(M, \partial M)$ (\cite[Lemma 3.5]{Hatcher}, see also \cite[Section 1.3]{Wat09}).

For every primitive element $\alpha \in \pi_1(T) \subset \pi_1(M)$, the quotient $\pi_1(M) / \langle \langle \alpha \rangle \rangle$
is the fundamental group of a $3$-manifold denoted $M(\alpha)$, which is the manifold constructed as follows.  Choose a simple closed curve $\gamma \subset T$ such that $[\gamma] = \alpha$, and a homeomorphism $f\colon \partial(S^1 \times D^2) \rightarrow T$ sending $\{*\} \times \partial D^2$ to $\gamma$, and set $M(\alpha) = M \cup_f (S^1 \times D^2)$.  The manifold $M(\alpha)$ is called the Dehn filling of $M$ along the curve $\alpha$; a standard argument shows that $H_1(M(\alpha); \mathbb{Z})$ is finite whenever $\alpha \neq \lambda_M^{\pm 1}$ \cite[Lemma 1.5]{Wat09}. Consequently, as $\pi_1(M(\alpha))$ is a finitely generated group whose abelianization $H_1(M(\alpha); \mathbb{Z})$ is finite, it is not locally indicable. 

Therefore from the short exact sequence
\[ 1 \rightarrow \langle \langle \alpha \rangle \rangle \rightarrow \pi_1(M) \rightarrow \pi_1(M(\alpha)) \rightarrow 1
\]
and an application of Corollary \ref{quotient corollary}, we conclude:

\begin{proposition}
If $M$ is a knot manifold, then $E_\mathsf{lo}(\pi_1(M))$ is nonsmooth whenever $M$ admits a primitive element $\alpha \in \pi_1(T)$ distinct from $\lambda_M^{\pm1}$ such that $\pi_1(M(\alpha))$ is left-orderable.
\end{proposition}

While the technology for determining left-orderability of $\pi_1(M(\alpha))$ for an arbitrary $\alpha$ is not yet developed, there is a conjectural picture known as the L-space conjecture.

\begin{conjecture}(\cite[Conjecture 1]{BGW13}, \cite[Conjecture 5]{Juh15})
\label{Lspace conjecture}
Suppose that $M$ is a compact, connected, orientable irreducible rational homology $3$-sphere that is not homeomorphic to $S^3$.  Then the following are equivalent:
\begin{enumerate}
\item $M$ is not a Heegaard Floer homology L-space;
\item $M$ admits a co-orientable taut foliation;
\item $\pi_1(M)$ is left-orderable.
\end{enumerate}
\end{conjecture}

The behaviour of Dehn filling with respect to Heegaard Floer homology is controllable,  in the sense that
\[ \mathcal{L}_M = \{ \alpha \in \pi_1(T) \mid \alpha \mbox{ primitive and $M(\alpha)$ is an L-space} \}
\]
is fairly well-understood.  In particular, we make the following observation:  Owing to the short exact sequence 
\[ 1 \rightarrow \langle \langle \alpha \rangle \rangle \rightarrow \pi_1(M) \rightarrow \pi_1(M(\alpha)) \rightarrow 1
\]
and Corollary~\ref{quotient corollary}, the L-space conjecture predicts that $E_\mathsf{lo}(\pi_1(M))$ is nonsmooth whenever there exists $\alpha \notin \mathcal{L}_M$ distinct from $\lambda_M^{\pm 1}$. The manifolds for which such an $\alpha$ fails to exist are known as generalized solid tori (\cite[Proposition 7.5]{RR17}, c.f. \cite[Section 1.5]{HRW22}). That is, a generalized solid torus is a knot manifold $M$ such that if $\alpha \neq \lambda_M$ is primitive, then $M(\alpha)$ is an L-space; examples of such manifolds are the solid torus itself and the  twisted $I$-bundle over the Klein bottle.  

We therefore conjecture:

\begin{conjecture}
\label{3m conjecture}
If $M$ is a knot manifold and $E_\mathsf{lo}(\pi_1(M))$ is smooth, then $M$ is a generalized solid torus. 
\end{conjecture}

It should be noted that $E_\mathsf{lo}(\pi_1(M))$ is smooth for the generalized solid tori listed above: The fundamental group of the solid torus is the integers, so $E_\mathsf{lo}(\pi_1(M))$ is smooth in this case, and the fundamental group of the twisted $I$-bundle over the Klein bottle is $ \langle x, y \mid xyx^{-1} = y^{-1} \rangle$, which admits only four orderings as in Example \ref{Klein}. So $E_\mathsf{lo}(\pi_1(M))$ is smooth in this case as well.

On the other hand, the set of $3$-manifolds that are generalized solid tori is conjecturally the same as the set of all LO-generalized solid tori, introduced in \cite[Page 24]{BGH}.  An LO-generalized solid torus is a knot manifold $M$ such that the restriction of every left-ordering of $\pi_1(M)$ to $\pi_1(T) \subset \pi_1(M)$ yields a lexicographic ordering of $\pi_1(T)$ arising from the short exact sequence
\[ 1 \rightarrow \langle \lambda_M \rangle \rightarrow \pi_1(T) \rightarrow \mathbb{Z} \rightarrow 1.
\]
The conjectured relationship with generalized solid tori comes from the fact that for these manifolds, $\pi_1(M(\alpha))$ is not left-orderable whenever $\alpha \neq \lambda_M$.

While $E_\mathsf{lo}(\pi_1(M))$ is easily understood for the solid torus and the Klein bottle, in \cite[Proposition 7.2]{BGH} the authors show that the hyperbolic knot manifold $M  = v2503$ is an example of an LO-generalized solid torus, and so is conjecturally also a generalized solid torus. For this manifold, the complexity of $E_\mathsf{lo}(\pi_1(M))$ is completely unknown.  

\begin{question}
    Determine the complexity of $E_\mathsf{lo}(\pi_1(M))$ when for $M  = v2503$.
\end{question}

We are able to verify Conjecture \ref{3m conjecture} for the case where $M$ is the complement of a knot in $S^3$, which in some sense is the simplest class of knot manifolds (in general, if $M$ is a closed $3$-manifold and $K$ is a knot in $M$ which does not bound a disk and is not contained in a $3$-ball, then the complement of $K$ in $M$ is a knot manifold). However to verify Conjecture \ref{3m conjecture} even in this special case, we need a generalization of Corollary \ref{quotient corollary}, as even in this simplest case one cannot guarantee the existence of an appropriate Dehn filling with left-orderable fundamental group.
\begin{proposition}
\label{restriction trick}
Suppose that $G$ is a countable left-orderable group, and that $H \subset G$ is a subgroup.   If there exists $P \in \LO(G)$ such that:
\begin{enumerate-(i)}
\item 
$\{ gPg^{-1} \cap H \mid g \in G\}$ is a closed subset of $\LO(H)$, 
\item
\label{item ii : restriction trick}
there is no $Q \in \CLO(G)$ and $g \in G$ satisfying $Q \cap H = gPg^{-1} \cap H$
 \end{enumerate-(i)}
 then $E_\mathsf{lo}(G)$ is not smooth.
\end{proposition}
\begin{proof}
First observe that the restriction map $r\colon \LO(G) \rightarrow \LO(H)$ is continuous, and consider the closed, $G$-invariant set $\overline{\{gPg^{-1} \mid g \in G\}}$. 

Let $Q \in \overline{\{gPg^{-1} \mid g \in G\}}$.  By Proposition~\ref{prop : LO minus CO} it suffices to show that \(Q\) is not Conradian.  First, if $Q =gPg^{-1}$ for some $g \in G$ then $Q \notin \CLO(G)$ by~\ref{item ii : restriction trick}.  On the other hand, if $Q$ is an accumulation point of $\{ gPg^{-1} \mid g \in G\}$, then since $\LO(G)$ and $\LO(H)$ are metrizable and $r$ is continuous, $r(Q)$ is an accumulation point of $\{ gPg^{-1} \cap H \mid g \in G\}$.  But as this set is closed, we must have $Q \cap H =gPg^{-1} \cap H$ for some $g \in G$, which again forces $Q \notin \CLO(G)$ by~\ref{item ii : restriction trick}.  
\end{proof}

We recall the basics of knot groups.  Recall that given a knot $K \subset S^3$, the knot group of $K$ is the fundamental group $\pi_1(S^3 \setminus N (K))$, where $N(K)$ denotes a tubular neighbourhood of $K$.  As $\partial(N(K))$ is a torus, $\pi_1(\partial(N(K))) \cong {\mathbb{Z}\oplus\mathbb{Z}}$, and it is a classical fact that inclusion $i\colon \partial(N(K)) \rightarrow S^3 \setminus N (K)$ induces an injective homomorphism $i_*\colon \pi_1(\partial(N(K))) \rightarrow \pi_1(S^3 \setminus N(K))$.   To simplify notation, we write $\pi (K)$ for the knot group of $K$. It is straightforward to check that the first homology group $H_1(S^3 \setminus N (K)) \cong \mathbb{Z}$, as such there is an abelianization map $\phi_{ab}\colon \pi (K) \rightarrow \mathbb{Z}$ for every knot $K$ in $S^3$.  The subgroup $i_*( \pi_1(\partial(N(K))))$ admits a choice of basis $\{ \mu, \lambda \}$ satisfying $\phi_{ab}(\mu) = 1$ and $\phi_{ab}(\lambda) = 0$.

\begin{lemma}
\label{conrad extension}
Suppose that $K$ is a knot in $S^3$, that $p, q$ are integers with $p \neq 0$,
and that $<$ is any left-ordering of $\pi(K)$.  If there exists $g \in \pi (K)$ such that $(\mu^p \lambda^q)^n < g$ for all $n \in \mathbb{Z}$, then $<$ is not Conradian.
\end{lemma}
\begin{proof}
Suppose that $<$ is a Conradian left-ordering of $\pi(K)$.  As $\pi(K)$ is finitely generated, there exists a maximal $<$-convex subgroup $C$ that is normal in $\pi(K)$ and such that $\pi(K)/C$ is abelian and the natural quotient ordering of $\pi(K)/C$ is Archimedean \cite[Proposition 9.19]{ClaRol}.   As such, the quotient map $\pi(K) \rightarrow \pi(K)/C$ can be identified with the abelianization map $\phi_{ab}\colon \pi (K) \rightarrow \mathbb{Z}$, and the inherited ordering of $\pi(K)/C$ can be identified with the natural ordering of $\mathbb{Z}$ (up to possibly reversing the ordering) \cite[Problem 9.20]{ClaRol}.  Therefore if $<$ is a Conradian ordering of $\pi(K)$, then $\phi_{ab}\colon \pi (K) \rightarrow \mathbb{Z}$ is an order-preserving map, in the sense that $g<h$ implies $\phi_{ab}(g) \leq \phi_{ab}(h)$ for all $g, h \in \pi(K)$.

Therefore if $(\mu^p \lambda^q)^n < g$ for all $n \in \mathbb{Z}$, then $\phi_{ab}((\mu^p \lambda^q)^n) = np \leq \phi_{ab}(g)$ for all $n \in \mathbb{Z}$.  This cannot happen.
\end{proof}

We next need a detailed analysis of the possible orderings of $\mathbb{Z} \oplus \mathbb{Z}$ in order to use the results of \cite{BGH}. 

Identifying $\mathbb{Z} \oplus \mathbb{Z}$ with the integer lattice points in $\mathbb{R}^2$ and choosing $a = (a_1, a_2) \in \mathbb{R}^2 \setminus \{(0,0)\}$, there are two canonical positive cones in $LO(\mathbb{Z} \oplus \mathbb{Z})$ determined by $a$ if $\frac{a_1}{a_2}$ is irrational, they are:
\[ P_a^+ = \{ (m,n) \mid a_1m + a_2n >0 \} \mbox{ and }  P_a^- = \{ (m,n) \mid a_1m + a_2n <0 \}.
\]
If $\frac{a_1}{a_2}$ is rational, then there are four associated positive cones, they are:
\[P_a^{++}= \{ (m,n) \mid a_1m + a_2n >0 \mbox{ or there exists $c>0$ such that } (-a_2, a_1) = (cn,cm) \}
\]
\[P_a^{+-}= \{ (m,n) \mid a_1m + a_2n >0 \mbox{ or there exists $c<0$ such that } (-a_2, a_1) = (cn,cm) \}
\]
\[P_a^{-+}= \{ (m,n) \mid a_1m + a_2n <0 \mbox{ or there exists $c>0$ such that } (-a_2, a_1) = (cn,cm) \}
\]
\[P_a^{--}= \{ (m,n) \mid a_1m + a_2n <0 \mbox{ or there exists $c<0$ such that } (-a_2, a_1) = (cn,cm) \}.
\]
Moreover, every $P \in \LO(\mathbb{Z} \oplus \mathbb{Z})$ arises from a choice of $a \in \mathbb{R}^2$ and one of these constructions above.  Thus, every $P \in \LO(\mathbb{Z} \oplus \mathbb{Z})$ has associated to it a \emph{slope}, which is the element $\alpha = [(-a_2, a_1)] \in \mathbb{P}(\mathbb{R}^2)$--we say that $P$ \emph{detects the slope $\alpha$} (note that this is precisely the slope of the line that divides $\mathbb{R}^2$ into two halves, one containing only positive elements of $\mathbb{Z} \oplus \mathbb{Z}$ and the other containing only negative elements of $\mathbb{Z} \oplus \mathbb{Z}$). Moreover, we are taking the equivalence class of $(-a_2, a_1)$ in the projectivization of $\mathbb{R}^2$ since each $P \in \LO(\mathbb{Z} \oplus \mathbb{Z})$ can be written in exactly two different ways, e.g. as $P_a^+$ and $P_{-a}^-$ for some choice of $a \in \mathbb{R}^2$.

\begin{lemma}
\label{convex knot lemma}
Suppose that $P \in \LO(\mathbb{Z} \oplus \mathbb{Z})$ detects $[(p,q)]$, where $p,q$ are integers.  Then in the associated ordering $<$ of $\mathbb{Z} \oplus \mathbb{Z}$, the subgroup $\langle (p,q) \rangle$ is convex.
\end{lemma}
\begin{proof}
It suffices to show that if $(0,0) < (a,b) < (np, nq)$ for some $n \in \mathbb{Z}$ then $(a, b) = (kp, kq)$ for some $k \in \mathbb{Z}$.  To show this, suppose that $P = P_{(q, -p)}^{++}$, the other cases are similar.

Then $(a, b) \in P$ implies $aq-pb>0$, or there exists $c >0$ such that $(p,q) = c(a,b)$.   If the latter holds, we are done.   The former, on the other hand, leads to a contradiction:  Since $(a,b) < (np, nq)$  also holds, then $(np-a, nq-b) \in P$, which implies $(np-a)q-(nq-b)p >0$ or $pb-aq >0$.  This proves the lemma.
\end{proof}

Now, given a knot $K \subset S^3$, let $H$ denote the subgroup $\langle \mu, \lambda \rangle \subset \pi(K)$ and $r\colon \LO(\pi(K)) \rightarrow \LO(H)$ denote the restriction map.  The authors of~\cite{BGH} define a slope $\alpha \in \mathbb{P}(H \otimes \mathbb{R})$
%\todo{I guess this is just a type and we mean \(H\oplus \mathbb{R}\)--No, here we need to tensor with R in order to allow "irrational slopes", we're looking to replace the integer lattice H with R^2}
to be \emph{order-detected} if there exists a positive cone $P \in \LO(\pi(K))$ such that $r(gPg^{-1})$ detects $\alpha$ for all $g \in \pi(K)$.  Note that the subgroup $H$ is isomorphic to $\mathbb{Z} \oplus \mathbb{Z}$, so $\mathbb{P}(H \otimes \mathbb{R})$ can be identified with the projectivization of $\mathbb{R}^2$ and $H$ can be identified with the integer lattice points, meaning that a typical slope arising from one of these lattice points is of the form $[\mu^p \lambda^q]$ for $p, q \in \mathbb{Z}$.

A \emph{Seifert surface} for a knot $K \subset S^3$ is an orientable surface $\Sigma \subset S^3$ such that $\partial \Sigma = K$.  Denoting the genus of a surface $\Sigma$ by $g(\Sigma)$, we define the genus of a knot $K$ to be 
\[ g(K ) = \min \{ g(\Sigma) \mid \mbox{$\Sigma$ is a Seifert surface for $K$}\}.
\]
The only knot $K$ for which $g(K) = 0$ is the unknot, for all other knots $g(K) >0$.

\begin{proof}[Proof of Theorem \ref{knot_result}]
According to \cite[Corollary 1.4]{BGH}, if $K$ is a nontrivial knot in $S^3$, then the slope $[\mu^{2g(K)-1} \lambda]$
is order-detected.  Correspondingly, using $H$ to denote the subgroup $\langle \mu, \lambda \rangle \subset \pi(K)$ and $r\colon \LO(\pi(K)) \rightarrow \LO(H)$ the restriction map, there is a positive cone $P \subset \pi(K)$ such that $r(gPg^{-1})$ detects $[\mu^{2g(K)-1} \lambda]$ for all $g \in \pi(K)$.  But then if we write $a = \mu^{-1} \lambda^{2g(K)-1}$ this means $r(gPg^{-1}) \in \{  P_a^{++}, P_a^{+-}, P_a^{-+}, P_a^{--}\}$ for all $g \in \pi(K)$, so that $r(\{ gPg^{-1} \mid g \in \pi(K)\})$ is finite, and hence closed in $\LO(H)$.

Moreover, suppose $Q \in \LO(\pi(K))$ satisfies $r(Q) \in  \{  P_a^{++}, P_a^{+-}, P_a^{-+}, P_a^{--}\}$. Then by Lemma \ref{convex knot lemma} the ordering associated to $Q$ must satisfy $(\mu^{2g(K)-1} \lambda)^n < \mu$ for all $n \in \mathbb{Z}$, since $\langle \mu^{2g(K)-1} \lambda \rangle $ is convex with respect to the ordering determined by $r(Q)$ and $\mu \notin \langle \mu^{2g(K)-1} \lambda \rangle$.  Since $K$ is nontrivial, $g(K)>0$ and so $2g(K) -1 \neq 0$. It therefore follows from Lemma \ref{conrad extension} that $Q \notin \CLO(\pi(K))$.

That $E_\mathsf{lo}(\pi(K))$ is not smooth now follows from Proposition \ref{restriction trick}.
\end{proof}

\section{Short exact sequences}

\label{sec : ses}

One of the main tools used in \cite{CalCla22}, and in many arguments in this paper, is short exact sequences of left-orderable groups.
In this section we give a general analysis of such methods in terms of equivariant maps. 

Suppose that 
\[1 \rightarrow K \stackrel{i}{\rightarrow} G \stackrel{q}{\rightarrow} H \rightarrow 1
\]
is a short exact sequence of left-orderable groups.  Then both $\mathrm{LO}(K)$ and $\mathrm{LO}(H)$ come equipped with a natural $G$-action, as follows.  If $g \in G$ and  $P \in \mathrm{LO}(H)$, define $g \cdot P = q(g)Pq(g)^{-1}$.  On the other hand if $P \in \mathrm{LO}(K)$, we will use $\varphi_g \in Aut(K)$ to denote the automorphism of $K$ defined by $\varphi_g(k) = gkg^{-1}$ and set $g \cdot P = \varphi_g(P)$.

Recall that a subgroup $H$ of $G$ is called \emph{absolutely convex} if $H$ is convex relative to every left-ordering of $G$.  

\begin{proposition}
\label{ses proposition}
Suppose that 
\[1 \rightarrow K \stackrel{i}{\rightarrow} G \stackrel{q}{\rightarrow} H \rightarrow 1
\]
is a short exact sequence of countable groups, and set $X = \mathrm{LO}(K) \times \mathrm{LO}(H)$.  If we equip $X$ with the diagonal $G$-action defined by 
\[ g \cdot (P_K, P_H) = (g \cdot P_K, g \cdot P_H)
\]
 and let $E^X_G$ denote the resulting equivalence relation, then $E^X_G \leq_B E_\mathsf{lo}(G)$.  If $K$ is absolutely convex, then $E^X_G \cong_B E_\mathsf{lo}(G)$.
\end{proposition}
\begin{proof}
First, we note that since the actions of $G$ on $\LO(H)$ and $\LO(K)$ are by homeomorphisms, each of $E^{\LO(H)}_G$ and $E^{\LO(K)}_G$ is a countable Borel equivalence relation.  Since all our groups are countable, $E^X_G$ is also a countable Borel equivalence relation.

Next, define a map $\theta\colon X \rightarrow \LO(G)$ by $\theta(P,Q) = i(P) \cup q^{-1}(Q)$.  Note that $\theta$ is continuous and hence Borel, in fact this is a Borel reduction.  This follows from the observation that $g \cdot (P, Q) = (P'Q')$ if and only if $gPg^{-1} = P'$ and $q(g)Qq(g)^{-1} = Q'$, which happens if and only if $$g(i(P) \cup q^{-1}(Q))g^{-1} = gi(P)g^{-1} \cup gq^{-1}(Q)g^{-1} = i(P') \cup q^{-1}(Q').$$

If $K$ is absolutely convex, then every positive cone $P_G \in LO(G)$ is of the form $i(P_K) \cup q^{-1}(P_H)$, and so the the map $\theta$ is surjective.  In this case, $\theta$ admits an inverse which is also a Borel reduction.
\end{proof}

We next make a general observation to be used in combination with the previous result.

\begin{proposition}
\label{product}
Suppose that $X, Y$ are each equipped with a Borel $G$-action for some countable group $G$.  Let $E^{X \times Y}_G$ denote the countable equivalence relation arising from the diagonal action $g \cdot(x,y) = (g \cdot x, g \cdot y)$ on $X \times Y$, assume that it is Borel.  If there exists a $G$-equivariant Borel map $\theta \colon X \rightarrow Y$ then $E^X_G \leq_B E^{X \times Y}_G$ .
\end{proposition}
\begin{proof}
Define a Borel reduction $\phi : X \rightarrow X \times Y$ by $\phi(x) = (x, \theta(x))$.  Suppose that $g \cdot x_1 = x_2$ for some $x_1, x_2 \in X$.  Then $$g \cdot \phi(x_1) = g \cdot (x_1, \theta(x_1)) = (g \cdot x_1, g \cdot \theta(x_1)) = (x_2, \theta(g \cdot x_1)) = (x_2, \theta( x_2)) = \phi(x_2).$$

Conversely if $g \cdot \phi(x_1) = \phi(x_2)$ then $(g \cdot x_1, g\cdot\theta(x_1)) = (x_2, \theta(x_2))$ and so $g \cdot x_1 = x_2$ from the first component.
\end{proof}

\begin{proposition}
Suppose that 
\[1 \rightarrow K \stackrel{i}{\rightarrow} G \stackrel{q}{\rightarrow} H \rightarrow 1
\]
is a short exact sequence of left-orderable groups.  
\begin{enumerate-(1)}
\item If there exists a $G$-equivariant map $\theta \colon \mathrm{LO}(K) \rightarrow \mathrm{LO}(H)$ then $E_\mathsf{lo}(K) \leq_B E_\mathsf{lo}(G)$, and
\item If there exists a $G$-equivariant map $\theta \colon \mathrm{LO}(H) \rightarrow \mathrm{LO}(K)$ then $E_\mathsf{lo}(H) \leq_B E_\mathsf{lo}(G)$. 
\end{enumerate-(1)}
\end{proposition}
\begin{proof}
This follows immediately from Proposition \ref{ses proposition} and Proposition \ref{product}.
\end{proof}

\begin{corollary}
Suppose that 
\[1 \rightarrow K \stackrel{i}{\rightarrow} G \stackrel{q}{\rightarrow} H \rightarrow 1
\]
is a short exact sequence of left-orderable groups. Then:
\begin{enumerate-(1)}
\item If $K$ admits positive cone $P$ such that $gPg^{-1} = P$ for all $g \in G$, then $E_\mathsf{lo}(H) \leq_B E_\mathsf{lo}(G)$.
\item If $H$ is bi-orderable then $E_\mathsf{lo}(K) \leq_B E_\mathsf{lo}(G)$.
\end{enumerate-(1)}
\end{corollary}
\begin{proof}
In both cases above we can apply the previous proposition taking $\theta$ to be a constant map.
\end{proof}

A corollary of this last proof already appears as \cite[Proposition~3.4]{CalCla22}, where the $G$-equivariant map used in that proposition is chosen to be a constant map.

\end{document}